\documentclass[11pt, twoside]{article}
\usepackage[english]{babel}

\usepackage{amsmath,amsthm,amscd,latexsym, eucal,epsfig,enumitem,bbold,color}

\usepackage{amssymb}
\usepackage{float} 
\usepackage{graphicx} 
\usepackage{pstricks,pstricks-add,pst-math,pst-xkey}
\usepackage{multicol}
\usepackage[labelformat=simple]{subcaption}

\begin{document}

\newtheorem{theorem}{Theorem}
\newtheorem{mainResult}[theorem]{Main result}
\newtheorem{proposition}[theorem]{Proposition}
\newtheorem{lemma}[theorem]{Lemma}
\newtheorem{corollary}[theorem]{Corollary}

\theoremstyle{definition}
\newtheorem{definition}[theorem]{Definition}
\newtheorem{notation}[theorem]{Notation}
\newtheorem{question}[theorem]{Question}
\newtheorem{remark}[theorem]{Remark}
\newtheorem{exempl}[theorem]{Exempl}
\providecommand{\keywords}[1]{\textbf{\textit{Keywords:}} #1}
\def\ONE{{\mathbb{1}}}

\title{Synchronization hypothesis in the Winfree model}

\author{W. Oukil$^{1}$, A.  Kessi$^{1}$ and Ph. Thieullen$^{2}$, \\ \quad \\
\small\text{1. Laboratory of Dynamical Systems (LSD), Faculty of Mathematics,}\\
\small\text{ University of Sciences and Technology Houari Boumediene,}\\
\small\text{ BP 32 El Alia 16111, Bab Ezzouar, Algiers, Algeria.}\\
\small\text{2. Institute of Mathematics of Bordeaux (IMB), University of Bordeaux,}\\
\small{351, cours de la Liberation - F 33 405 Talence, France.}}

\date{\today}

\maketitle
\begin{abstract}
We consider $N$ oscillators coupled by a mean field as in the Winfree model. The model is governed by two parameters: the coupling strength $\kappa$ and the spectrum width $\gamma$ of the frequencies of each oscillator centered at 1.  In the uncoupled regime, $\kappa=0$,  each oscillator possesses its own natural frequency, and the difference between the phases of any two oscillators grows linearly in time. In the zero-width regime for the spectrum, the oscillators are simultaneously in the death state if and only if $\kappa$ is above some positive value $\kappa_*$.  We say that $N$ oscillators are synchronized if the difference between  any two phases is uniformly bounded in time. We identify a new hypothesis for the existence of synchronization. The domain in $(\gamma,\kappa)$ of synchronization contains $\{0\} \times [0,\kappa_*]$ in its closure.  Moreover the domain is independent of the number of oscillators and the distribution of the frequencies. We show numerically  that the above hypothesis is necessary for the existence of synchronization.  
\end{abstract}



\section{Introduction}

In 1967 Winfree \cite{WinfreeModel}   proposed a model  describing the synchronization of a population of organisms or {\it oscillators} that interact simultaneously. We assume that the state of each oscillator is described by a real number, the phase, that corresponds to a  point on the circle. We call natural frequency, the frequency of each oscillator, as  if it were isolated from the others. The natural frequencies are  localized inside an interval $[1-\gamma,1+\gamma]$ for some constant $\gamma$ that we call {\it spectrum width}. The interaction of the surrounding oscillators on a particular oscillator has a shape independent of this oscillator. The term {\it mean field} is used for this kind of interaction or coupling. We denote by $N$ the number of oscillators and by $\kappa$ a parameter called {\it coupling strength} that measures the coupling between these oscillators.

A numerical investigation and a mathematical analysis of the  transitions between the phases  show a diagram of 4 major states (see  Ariaratnam and Strogatz \cite{AriaratnamStrogatz}).  A phase called {\it death state} where each oscillator is at rest, a phase called {\it locking state} where all the oscillators possess the same non-zero frequency, a phase called {\it  incoherence state} where  any two oscillators have different frequencies, and a phase called {\it  partial death state} which is a mixed state where a  part of the oscillators are in death state and a  part is in incoherence state.  For small values  of $\kappa$, the Winfree model is reduced to a different model, the Kuramoto model that will not be studied in this article. For large values of $\kappa$, the oscillators are in the death state. We denote by $\kappa_*$ the bifurcation parameter for the death state in the case $\gamma=0$.

We say that $N$ oscillators are synchronized if $\max_{1\leq i, j\leq N} |x_i(t) - x_j(t)|$ is bounded from above uniformly in time. Notice that we do not consider the phase of the oscillators but rather their lifts. For example, in the case $\gamma=0$, our definition implies that $N$ oscillators are always synchronized, independently of their initial conditions,  thanks to the uniqueness property of periodic solutions of an ODE. The existence of synchronization has been studied in \cite{Kuramoto,Omelchenko,Panaggio} for $\kappa$ close to 0. We identify a new hypothesis of synchronization that produces  a domain of parameters  $(\gamma,\kappa)$  containing $\{0\} \times [0,\kappa_*]$ in its closure.  We show in addition in a numerical example that the existence of synchronization  disappears precisely when the hypothesis is not anymore realized.

The Winfree model is given by the following differential equation
\begin{equation}\label{equation:WinfreeModel}
\dot{x}_{i}=\omega_i-\kappa  \frac{1}{N}\sum_{j=1}^{N}P(x_j)R(x_{i})
 \end{equation}
where $P$ and $R$ are $C^2$ $2\pi$-periodic functions, $x_i(t)$ is the phase of the i-th oscillator, and  $X(t)=(x_1(t), \ldots,x_N(t))$ is the global state of the system. Although $x_i(t)$ represents a scalar in the torus $\mathbb{R}/2\pi \mathbb{Z}$, we actually consider its unique continuous lift in $\mathbb{R}$, that we continue to call it  $x_i(t)$. When $N=1$ and $\omega_1=1$, the Winfree model reduces to the equation
\begin{equation} \label{equation:SimpleWinfreeModel}
\dot x = 1 -\kappa P(x)R(x).
\end{equation}
We call {\it locking bifurcation critical parameter}, the coupling critical parameter $\kappa_*$ which separates the death state and the locking state in the reduced Winfree model:
\begin{equation}
\label{equation:DeafBifurcationParameter}
\kappa_* := \max\{ \kappa>0 \,:\, 1-\kappa P(x) R(x) >0, \ \forall x\in\mathbb{R} \}.
\end{equation}
(Notice that $\kappa_*=+\infty$ if $\max_x P(x)R(x)\leq 0$.) We assume through out this work the following hypotheses:
\begin{itemize}
\item[H1] the coupling strength $\kappa$ is taken in the interval $ (0,\kappa_*)$,
\item[H2] the natural frequencies  $\omega_i$ are chosen in $(1-\gamma,1+\gamma)$ with $\gamma \in (0,1)$,
\item[H3] synchronization hypothesis: \ $\displaystyle{\int_0^{2\pi}\! \frac{P(s)R'(s)}{1-\kappa P(s)R(s)} \,ds} > 0$, $\forall \kappa \in(0,\kappa_*)$.
\end{itemize} 
For numerical results we use a simplified version of the Winfree model
\begin{equation}\label{equation:NumericalWinfreeModel}
\dot x_i = \omega_i - \frac{\kappa}{N} \sum_{j=1}^N P_\beta(x_j) \sin(x_i), \ \text{where} \ P_\beta(x) = 1+\cos(x+\beta),
\end{equation}
and $\beta \in [0,\pi]$. It is important to notice that we do not make any assumptions on the number $N$ of oscillators; in particular we do not assume that $N \to +\infty$. We do  not assume either that the natural frequencies are distributed according to  a particular law as it is done in \cite{AriaratnamStrogatz}. Our main objective is to exhibit a domain in the parameter space $(\gamma,\kappa)$ independent of $N$ and any choice of the natural frequencies satisfying H2 such that the oscillators stay synchronized for all time and  for any initial conditions sufficiently packed.

Synchronization and locking may have several meanings or definitions depending on the authors. We choose the following definitions.
\begin{definition}[Synchronization]
We say that the oscillators $x_i(t)$ are synchronized  if $\sup_{1 \leq i,j \leq N}  |x_i(t)-x_j(t)| $ is bounded from above uniformly in time $t\ge0$.
\end{definition}

\begin{definition}[Periodic locking]\label{definition:periodiclocking}
We say that the  oscillators $x_i(t)$ are periodically locked to the frequency $\Omega$ if they are synchronized and if there exist $2\pi/\Omega$-periodic functions $\Psi_i(t)$  such that
\[
x_i(t) = \Omega t + \Psi_i(t), \quad \forall i= 1\ldots N,  \ \forall t\geq0.
\]
\end{definition}
\noindent We notice that in the periodic locking case  the average frequency $\frac{x_i(t)}{t}$ admits the limit $\Omega$ as $t\to+\infty$ and that this limit is independent of the ith oscillator. The fact that the limit does exist has never been addressed mathematically (except of course in the death state). Our main result is a partial result in that direction in the locking case when $\kappa$ is any parameter in $(0,\kappa_*)$  and $\gamma  \in (0, \Gamma(\kappa))$ for some $\Gamma : [0,\kappa_*] \to \mathbb{R}^+$.

\begin{theorem}\label{mainresult}
We consider the Winfree model given by \eqref{equation:WinfreeModel} and satisfying the hypotheses H1--H3. Then, there exists an open set $U$ in the space of parameters $(\gamma,\kappa) \in (0,1) \times (0,\kappa_*)$, independent of $N$, containing in its closure $\{0\}\times [0,\kappa_*]$,  such that for every parameter $(\gamma,\kappa) \in U$, for every $N\geq1$ and  every choice of natural frequencies $(\omega_i)_{i=1}^N$,
\begin{enumerate}
\item There exists an open set $C_{\gamma,\kappa}^N$ invariant by the flow, and of the form, 
\[
C_{\gamma,\kappa}^N := \Big\{ X=(x_i)_{i=1}^N \in \mathbb{R}^N \,:\, \max_{i,j}|x_j-x_i| < \Delta_{\gamma,\kappa}\Big( \frac{1}{N}\sum_{i=1}^Nx_i \Big) \Big\}
\]
where $\Delta_{\gamma,\kappa} : \mathbb{R} \to (0,1)$ is a $2\pi$-periodic $C^2$ function independent of $N$. In other words, for every parameters  $(\gamma,\kappa) \in U$,  the oscillators are synchronized for any initial conditions $x_i(0)$ with sufficiently small dispersion.
\item  There exists a particular  initial condition $(x_i^*(0))_{i=1}^N\in C_{\gamma,\kappa}^N$, and a common frequency  $\Omega_{\gamma,\kappa}>0$ such that
\[
x_i^*(t) = \Omega_{\gamma,\kappa} t + \Psi_{i,\gamma,\kappa}^N(t), \quad \forall i=1,\ldots,N, \ \forall  t\geq 0,
\]
where $\Psi_{i, \gamma,\kappa}^N : \mathbb{R}^+ \to \mathbb{R}$ are $C^2$,  ${2\pi}/{\Omega_{\gamma,\kappa}}$-periodic functions, uniformly bounded with respect to $N$. In other words, the oscillators are periodically  locked for a particular initial condition.
\end{enumerate}
\end{theorem}
Notice that the simplified Winfree model \eqref{equation:NumericalWinfreeModel} satisfies the hypothesis H3 when $\beta =0$ as it is proved in section \ref{section:NumericalResults}.

For small values of $\kappa <<1$, Quinn and Strogatz \cite{Quinn} use the Lindstedt's Method and find an approximation value of the tangent bifurcation curve of the synchronization sate. We mention relative results about the Winfree model in \cite{Popovych, Giannuzzi,Basnarkov,Louca,PazoMontbrio}. Recently Ha, Park and Ryoo studied  in \cite{HaParkRyoo} the stability of the stationary state for large coupling.

\section{Proof of theorem \ref{mainresult}}\label{Reduction}

Let $X=(x_1,\ldots,x_N)$ be a solution of \eqref{equation:WinfreeModel}. We call mean $\mu(t)$ and dispersion $\delta(t)$ of $X$ the quantities
\begin{equation} \label{equation:mudelta}
\mu(t) := \small\frac{1}{N}\small\sum_{j=1}^{N}x_j(t), \quad  \delta(t) := \max_{i,j}|\delta_{i,j}(t)|, \quad  \delta_{i,j}(t) := x_i(t)-x_j(t).
\end{equation}
We want to build a space of parameters $U$ of the form
\[
U=\{ (\gamma,\kappa) \in (0,1) \times (0,\kappa_*) \,:\, 0 < \gamma < \Gamma(\kappa) \},
\]
and a $2\pi$-periodic function $\Delta_{\gamma,\kappa}$, called {\it dispersion curve} such that, if $X(t)$ is a solution of the Winfree model \eqref{equation:WinfreeModel} satisfying H1 -- H3,
\[
\delta(0) < \Delta_{\gamma,\kappa}(\mu(0)) \quad\Longrightarrow \quad \delta(t) < \Delta_{\gamma,\kappa}(\mu(t)), \quad \forall t \geq 0.
\]

The dispersion curve is obtained by solving a non-autonomous affine differential equation with periodic coefficients. The following lemma is standard. 

\begin{lemma} \label{lemma:linearWinfree}
We consider the affine differential equation
\begin{equation} \label{equation:linear Winfree}
\frac{d}{ds} \Delta(s) = \alpha -\beta(s) \Delta(s),
\end{equation}
where $\alpha>0$ and $\beta$ is a $C^1$, $2\pi$-periodic function satisfying 
\[
\int_0^{2\pi}\! \beta(s) \, ds >0.
\]
Then there exists a unique $C^2$, positive  and $2\pi$-periodic function solution of \eqref{equation:linear Winfree} given by
\[
\Delta(s) := \alpha \frac{\int_s^{s+2\pi} \exp \big( \int_s^t \beta(u) \, du \big) dt}{\exp \big(\int_0^{2\pi}\! \beta(u) \,du \big)-1}.
\]
Moreover
\[
\max_{s\in\mathbb{R}} \Delta(s) \leq  \alpha 2\pi \frac{\exp \big( \int_0^{2\pi}\! \beta^-(u) \,du \big)}{1-\exp \big(-\int_0^{2\pi}\! \beta(u) \,du \big)}
\]
where $\beta^- := \max(0,-\beta)$, and $\beta^+ := \max(0,\beta)$.
\end{lemma}

\begin{proof}

The differential equation \eqref{equation:linear Winfree}  is a first-order non-homogeneous  linear differential equation. The general form of solution $z(t)$ of \eqref{equation:linear Winfree} with initial condition $z(s_{0})\in \mathbb{R}$ at time $s_0\in \mathbb{R}$ is given by
\[
z(s)=\exp \big( \int_{s_{0}}^{s}\! -\beta(u) \,du \big)z(s_{0})+\alpha\int_{s_{0}}^{s} \exp \big( \int_t^s -\beta(u) \, du \big)dt.
\]
If $z(s)$ is $2\pi$-periodic, using the $2\pi$-periodicity of $\beta$, one obtains a unique $z(s_0)$ given by
\[
z(s_{0})=\alpha\frac{\int_{s_{0}}^{s_{0}+2\pi} \exp \big( \int_t^{s_{0}+2\pi} -\beta(u) \, du \big)dt}{1-\exp \big( \int_{s_{0}}^{s_{0}+2\pi}\! -\beta(u) \,du \big)}
\]
and a unique $2\pi$-periodic solution,  positive thanks to $\alpha>0$, given by
\[
z(s)=\alpha\frac{\int_{s}^{s+2\pi} \exp \big( \int_t^{s+2\pi} -\beta(u) \, du \big) dt}{1-\exp \big( \int_{0}^{2\pi}\! -\beta(u) \,du \big)}.
\]
The maximum of $\Delta$ is obtained using the following estimates
\begin{align*}
\int_s^t\! \beta(s) \,ds &= \int_s^t\! (\beta^+-\beta^-)(s) \,ds \leq \int_s^t\! \beta^+(s) \, ds \leq \int_0^{2\pi}\! \beta^+(s) \, ds, \\
&\leq \int_0^{2\pi}\! (\beta^+-\beta^-)(s) \,ds +  \int_0^{2\pi}\! \beta^-(s) \,ds, \\
&\leq  \int_0^{2\pi}\! \beta(s) \,ds +  \int_0^{2\pi}\! \beta^-(s) \,ds. \qedhere
\end{align*}
\end{proof}

The following lemma shows that, if  the dispersion  of a solution $X(t)$ of \eqref{equation:WinfreeModel} is a priori uniformly bounded, $\delta(t) < D, \ \forall t \in[0,t_*]$, then each $\delta_{i,j}(t)$ is a sub-solution strict of an affine differential equation as in lemma \ref{lemma:linearWinfree}.

\begin{lemma}\label{lemma:ODEdispersion}
Let us assume that for some $(\gamma,\kappa) \in (0,1) \times (0,\kappa_*)$, there exist $D>0$ and $t_*>0$ such that for every $t\in [0,t_*]$, the solution $X(t)$ of \eqref{equation:WinfreeModel} satisfies $\delta(t) < D$. Then for every $1 \leq i,j\leq N$, and $t\in[0,t_*]$
\begin{gather*}
\frac{d}{dt}\delta_{i,j} < (2\gamma+C\kappa D^2) -\kappa P(\mu)R'(\mu) \delta_{i,j}, \\
\big|\frac{d}{dt} \delta_{i,j}\big| \leq 2\gamma +C\kappa D^2 + \tilde C \kappa D,
\end{gather*}
where $\delta_{i,j}(t) := x_i(t)-x_j(t)$, $\mu(t) = \frac{1}{N}\sum_{k=1}^N x_k(t)$, 
\[
C :=\|P\|_\infty \|R''\|_\infty + \|P'\|_\infty \|R'\|_\infty, \  \text{and} \  \tilde C := \|P'\|_\infty \|R\|_\infty + \|P\|_\infty\|R'\|_\infty.
\]
\end{lemma}

\begin{proof}
By substracting two equations \eqref{equation:WinfreeModel} we obtain
\begin{align*}
\frac{d \delta_{i,j}}{dt} &=\omega_i -\omega_j -\frac{\kappa}{N} \sum_{k=1}^N P(x_k) \big[ R(x_i)-R(x_j) \big] \\
&= \omega_i - \omega_j - \kappa P(\mu) R'(\mu) \delta_{i,j} + \frac{\kappa}{N} \sum_{k=1}^N  E_{i,j,k}
\end{align*}
where $E_{i,j,k} = P(\mu)R'(\mu)(x_i-x_j) - P(x_k) [ R(x_i) - R(x_j) ]$ that we bound from above using the following estimates
\begin{gather*}
P(x_k) [ R(x_i) - R(x_j) ] =  P(\mu) [ R(x_i) - R(x_j) ] + \tilde E_{i,j,k}, \\
|\tilde E_{i,j,k}| \leq \|P'\|_\infty \|R'\|_\infty |x_k-\mu||x_i-x_j| < \|P'\|_\infty \|R'\|_\infty D^2, \\
P(\mu)  [ R(x_i) -R(x_j) ] = P(\mu) R'(\mu) (x_i-x_j) + \hat E_{i,j,k}, \\
|\hat E_{i,j,k}| < \|P\|_\infty \|R''\|_\infty D^2,
\end{gather*}
and $|\omega_i - \omega_j| \leq 2\gamma$.
\end{proof}

We estimate in the following lemma the velocity of $\mu$. We will later  find conditions on $(\gamma,\kappa)$ so that $\frac{d\mu}{dt}>0$. The constant $\tilde C$ is defined in lemma \ref{lemma:ODEdispersion}.

\begin{lemma} \label{lemma:ODEmean}
Assuming the same hypotheses as in lemma \ref{lemma:ODEdispersion}, we have
\[
\Big| 1-\kappa P(\mu)R(\mu) -\frac{d}{dt} \mu \Big| \leq \gamma + \tilde C \kappa D, \quad \forall t \in [0,t_*].
\]
\end{lemma}

\begin{proof}
We  use indeed the estimates $|x_i-\mu| \leq D$,  $|1-\frac{1}{N} \sum_{k=1}^N \omega_k | \leq \gamma$,
\begin{gather*}
\frac{d\mu}{dt} = 1  - \kappa P(\mu)R(\mu)  - \Big[ 1 - \frac{1}{N} \sum_{k=1}^N \omega_k \Big] + \frac{\kappa}{N} \sum_{k=1}^N E_k, \quad\text{where}\\
E_k = \frac{1}{N} \sum_{i=1}^N \Big[ P(\mu) \big( R(\mu) -R(x_i) \big) - \big( P(x_k) - P(\mu) \big) R(x_i) \Big]. \qedhere
\end{gather*}
\end{proof}

We now assume that $(\gamma,\kappa)$ have been chosen so that
\begin{equation}\label{equation:intermediaiteCondition}
1-\gamma-\tilde C \kappa D -\frac{\kappa}{\kappa_*} >0.
\end{equation} 
Notice that the equation $\frac{d x_i}{dt} = 1  - \kappa P(x_i)R(x_i) $ corresponds to  $\gamma=0$ and identical initial conditions in equation \eqref{equation:WinfreeModel}. By a small perturbation of $\gamma = 0$, and a small perturbation of the initial condition $x_1(0) = \cdots = x_N(0)$, one obtain $\frac{d\mu}{dt} = 1  \pm \epsilon - \kappa P(\mu)R(\mu) $ and $\frac{d}{dt}(x_i-x_j) = \pm\epsilon_{i,j}-\kappa P(\mu)R'(\mu)(x_i-x_j)$. The synchronization hypothesis H3 may be interpreted as a stability hypothesis after changing the time variable $t$ to $s=\mu(t)$.

We now proceed formally. Thanks to the definition of $\kappa_*$ in \eqref{equation:DeafBifurcationParameter},  condition \eqref{equation:intermediaiteCondition}  implies $\frac{d\mu}{dt}>0$ on $[0,t_*]$.  We then  consider the change of variable $s= \mu(t)$ from $[0,t_*]$ to $[s_0,s_*]$ whre $s_0=\mu(0)$ and $s_* = \mu(t_*)$. Let $\tau$ be the inverse function
\begin{equation}\label{eq:inverse_function}
\tau : \left\{\begin{array}{ccc}
[s_0,s_*] & \to & [0,t_*] \\ 
s & \mapsto & \tau (s)\\
\end{array}\right. .
\end{equation}
Define  $x_{i}^{*}(s)=x_i \circ \tau(s)$,  $X^*(s) =(x_1^*(s),\ldots,x_N^*(s))$,  $\delta_{i,j}^*(s) = x_i^*(s)-x_j^*(s)$.
 With respect to the new variable $s$, lemma \ref{lemma:ODEdispersion} admits the following equivalent form.

\begin{lemma}\label{delta}
We assume the same hypotheses as in lemma \ref{lemma:ODEdispersion} and the new condition \eqref{equation:intermediaiteCondition}. Then 
\begin{equation*}\label{delta_f}
\frac{d}{ds} \delta_{i,j}^*(s) < \alpha(\gamma,\kappa,D) - \beta_\kappa(s) \delta_{i,j}^*(s), \quad \forall s \in [s_0,s_*],
\end{equation*}
where 
\begin{gather*}
\alpha(\gamma,\kappa,D) := \frac{2\gamma +C \kappa D^2}{1-\kappa/\kappa_*} + \frac{(2\gamma+C\kappa D^2+\tilde C \kappa D)(\gamma+\tilde C \kappa D)}{(1-\gamma-\tilde C \kappa D -\kappa/\kappa_*)(1-\kappa/\kappa_*)}, \\
\beta_\kappa(s) := \frac{\kappa P(s) R'(s)}{1-\kappa P(s) R(s)}.
\end{gather*}
\end{lemma}

\begin{proof}
The chain rule gives $\frac{d}{dt} \delta_{i,j} = \frac{d}{ds} \delta_{i,j}^* \frac{d}{dt}\mu$, or
\[
\frac{d}{ds}  \delta_{i,j}^* \big(1-\kappa P(\mu)R(\mu) \big) = \frac{d}{dt} \delta_{i,j} + \frac{d}{ds} \delta_{i,j}^* \Big( 1-\kappa P(s) R(s) - \frac{d}{dt} \mu \Big).
\]
The definition of $\kappa_*$ implies $1-\kappa P(\mu)R(\mu) \geq 1 - \kappa/\kappa_*>0$. From lemma \ref{lemma:ODEdispersion}, one obtains
\[
\frac{d}{ds}  \delta_{i,j}^* \le \frac{(2\gamma+C\kappa D^2) + \frac{d}{ds} \delta_{i,j}^* \Big( 1-\kappa P(s) R(s) - \frac{d}{dt} \mu \Big)}{1-\kappa P(\mu)R(\mu) }-  \beta_\kappa(s) \delta_{i,j}.
\]
From lemma \ref{lemma:ODEmean}, one obtains
\begin{gather*}
\frac{d}{dt} \mu \geq 1 - \kappa P(\mu)R(\mu) -\gamma -\tilde C \kappa D \geq 1- \gamma  - \tilde C \kappa D -\kappa/\kappa_*, \\
|\frac{d}{ds} \delta_{i,j}^*| \leq \frac{2\gamma +C\kappa D^2 + \tilde C \kappa D}{1- \gamma  - \tilde C \kappa D -\kappa/\kappa_*}. \qedhere
\end{gather*}
\end{proof}

Let $\Delta_{\gamma,\kappa,D}(s)$ be the unique $C^2$ $2\pi$-periodic solution of the equation 
\[
\frac{d \Delta}{ds} = \alpha(\gamma,\kappa,D) -\beta_\kappa(s) \Delta
\]
given by lemma \ref{lemma:linearWinfree}. The following lemma gives  sufficient conditions on $(\gamma,\kappa,D)$ so that  $\max_s \Delta_{\gamma,\kappa,D}(s) < D$.

\begin{lemma}\label{lemma:definitionOpenSet}
There exists an open set $U$ of parameters $(\gamma,\kappa)$ whose closure contains $\{0\} \times [0,\kappa_*]$, defined in the following way
\begin{gather*}
U := \{ (\gamma,\kappa) \in\ (0,1) \times (0,\kappa_*) \,:\, 0 < \gamma < \kappa D^2(\kappa) \}, \quad\text{where} \\
D(\kappa) := \min \Big(1, \frac{ L(\kappa) }{2(2+C)/(1-\kappa/\kappa_*)+ 2\tilde C (1+ \tilde C) \kappa/(1-\kappa/\kappa_*)^2} \Big), \notag \\
L(\kappa) :=  \frac{1-\exp(-\int_0^{2\pi} \beta_\kappa(s) \,ds)}{2\pi \kappa \exp(\int_0^{2\pi} \beta_\kappa^-(s) \,ds)}, \label{equation:definitionOpenSet}
\end{gather*}
such that, for every $(\gamma,\kappa) \in U$,
\[
1-\gamma -\tilde C \kappa D(\kappa) -\frac{\kappa}{\kappa_*} >0, \quad\text{and}\quad \max_{s\in[0,2\pi]}\Delta_{\gamma,\kappa,D(\kappa)}(s) < D(\kappa). 
\]
\end{lemma} 

\begin{proof}
Thanks to lemma \ref{lemma:linearWinfree}, it is enough to check that
\[
\frac{\alpha(\gamma,\kappa,D(\kappa))}{\kappa  L(\kappa)} < D(\kappa).
\]
We have
\begin{gather*}
D(\kappa) < \frac{1-\kappa/\kappa_*}{2\kappa(1+\tilde C)}, \quad 
\gamma+\tilde C \kappa  D(\kappa) < \kappa D(\kappa) (1+\tilde C) < \frac{1}{2} \Big( 1-\frac{\kappa}{\kappa_*} \Big), \\
1- \gamma - \tilde C \kappa D(\kappa) -\kappa/\kappa_* >  \frac{1}{2} \Big( 1-\frac{\kappa}{\kappa_*} \Big), \\
\alpha(\gamma,\kappa,D(\kappa)) < \frac{2(2\gamma+C \kappa D(\kappa)^2)}{1-\kappa/\kappa_*} + \frac{\tilde C \kappa D(\kappa)
(\gamma+\tilde C \kappa D(\kappa))}{\frac{1}{2}(1-\kappa/\kappa_*)^2}, \\
< \kappa D(\kappa)^2 \Big[ \frac{2(2+C)}{1-\kappa/\kappa_*} + \frac{2\tilde C (1+\tilde C) \kappa}{(1-\kappa/\kappa_*)^2} \Big] < \kappa D(\kappa) L(\kappa). \qedhere
\end{gather*}
\end{proof}

\begin{definition}\label{dispersion_curve}
We call {\it  dispersion curve}  the periodic function 
\[
\Delta_{\gamma,\kappa}(s) : = \Delta_{\gamma,\kappa,D(\kappa)}(s), \quad \forall s \in \mathbb{R}.
\]
\end{definition}

We prove the first part of Theorem \ref{mainresult}. The set $C^N_{\gamma,\kappa}$ defined in Theorem \ref{mainresult} is obviously open and   has the shape of a tubular neighborhood about the line $(1,1,\ldots,1)\mathbb{R}$ with bounded convex transverse section $\mu(X)=\mu_0$. We want to prove that $C^N_{\gamma,\kappa}$ is invariant by the flow of \eqref{equation:WinfreeModel}. We recall that $U$ denotes the set of parameters $(\gamma,\kappa)$ defined in lemma \ref{lemma:definitionOpenSet}.

\begin{proof}[Proof of Theorem \ref{mainresult} - Item 1]
Let be $X(0)\in C_{\gamma,\kappa}^N$ and 
\[
t_* := \sup \{t \geq 0 : \forall\ 0< t' < t, \ \delta(t') <\Delta_{\gamma,\kappa}(\mu(t')) \}.
\]
Assume by contradiction that $t_* < +\infty$. We use the change of variable $s=\mu(t)$ for every $s\in [s_0,s_*]$, $s_0=\mu(0)$ and $s_*=\mu(t_*)$.  Then there exist $1 \le i_0, j_0\le N$ such that $\delta^*_{i_0,j_0}(s_*)=\Delta_{\gamma,\kappa}(s_*)$. Since $\max \Delta_{\gamma,\kappa} < D(\kappa)$ and $1-\gamma-\tilde C \kappa D(\kappa) - {\kappa}/{\kappa_*} >0$,  lemma \ref{delta} implies
\begin{align*}
\frac{d}{ds}\delta^*_{i_{0},j_{0}}(s_*) &< \alpha(\gamma,\kappa,D(\kappa)) - \beta_\kappa(s) \delta^*_{i_{0},j_{0}}(s_*), \\
&= \alpha(\gamma,\kappa,D(\kappa)) - \beta_\kappa(s) \Delta_{\gamma,k}(s_*)= \frac{d}{ds} \Delta_{\gamma,\kappa}(s_*).
\end{align*}
There exists $s< s_*$ close enough to $s_*$ such that $\delta^*_{i_{0},j_{0}}(s)> \Delta_{\gamma,\kappa}(s)$ or in other words there exists $t < t_*$ close enough to $t_*$ such that $\delta_{i_0,j_0}(t) > \Delta_{\gamma,\kappa}(\mu(t))$. We have obtained a contradiction. 
\end{proof}

We now prove the second part of Theorem \ref{mainresult}. We  show   there exists a periodically locked  solution $X(t)=(x_1(t),\ldots,x_N(t))$ of the system \eqref{equation:WinfreeModel} with some initial condition in $C_{\gamma,\kappa}^N$, that is, there exists a common rotation number $\Omega_{\gamma,\kappa}>0$ such that  $x_i(t)=\Omega_{\gamma,\kappa}t+\Psi_{i}(t)$, $\forall i=1,\cdots,N$ where $\Psi_i(t)$ are periodic functions of period $2\pi/\Omega_{\gamma,\kappa}$. Our strategy consists in constructing,  by fixing the mean of $X(0)$,  a compact and convex transverse section $\Sigma_{\gamma,\kappa}$ to the closure $\bar{C}_{\gamma,\kappa}^N$, and a continuous Poincar\'e map $P_{\gamma,\kappa} : \Sigma_{\gamma,\kappa} \to \Sigma_{\gamma,\kappa}$ by waiting the first time the mean of $X(t)$ return to 0. We then use Brouwer fixed point theorem to prove the existence of a fixed point of $P_{\gamma,\kappa}$. We denote by $\Phi^t(X)$ the flow of the equation \eqref{equation:WinfreeModel}.

\begin{lemma}\label{Poincare}
 Let be $(\gamma,\kappa) \in U$ where $U$ is defined in lemma \ref{lemma:definitionOpenSet}. Define
\[
\Sigma_{\gamma,\kappa}=\{X\in C_{\gamma,\kappa}^N \,:\, \mu(X)=0\} \quad\text{and}\quad \mathbb{1}=(1, \cdots, 1).
\]
Then there exist a $C^2$ map (the Poincar\'e map) $P_{\gamma,\kappa} : \Sigma_{\gamma,\kappa} \to \Sigma_{\gamma,\kappa}$ and a $C^2$ function (the return time map) $\theta_{\gamma,\kappa} : \Sigma_{\gamma,\kappa} \to \mathbb{R}^+$ such that, 
\begin{gather*}
\Phi^{\theta(X)}(X) = P_{\gamma,\kappa}(X) + 2\pi \mathbb{1}, \quad \forall X \in \Sigma_{\gamma,\kappa}, \\
\frac{2\pi}{1+\gamma+\kappa \|P\|_\infty \|R\|_\infty} < \theta(X) < \frac{2\pi}{1-\gamma - \tilde C \kappa D(\kappa)  -\frac{\kappa}{\kappa_*}}.
\end{gather*}
\end{lemma}

\begin{proof}
Let be $X \in C_{\gamma,\kappa}^N$ such that $\mu(X)=0$. Let be  $\mu(t) := \mu(\Phi^t(X))$ and $\tau(s)$ be the inverse function of $\mu(t)$ as it has been defined in \eqref{equation:mudelta} and \eqref{eq:inverse_function}. We prefer to write the explicit dependence on $X$: $\mu_X(t) = \mu(t)$ and $\tau_X(s) =\tau(s)$. Thanks to lemma \ref{lemma:ODEmean}, we obtain
\[
1-\gamma - \tilde C \kappa D(\kappa)- \frac{\kappa}{\kappa_*}   < \dot\mu_X(t) < 1+ \gamma + \kappa \|P\|_\infty \|R\|_\infty.
\]
Define $\theta(X) := \tau_X(2\pi)$. Then  $\int_0^{\tau_X(2\pi)} \dot\mu_X(t) dt = 2\pi$ implies the second estimate of the lemma. Let be $P_{\gamma,\kappa}(X) := \Phi^{\theta(X)}(X) - 2\pi \mathbb{1}$. Then
\begin{gather*}
\mu(P_{\gamma,\kappa}(X)) = \mu_X(\theta(X)) -2\pi = \mu_X \circ \tau_X(2\pi)) - 2\pi = 0, \\
\delta(P_{\gamma,\kappa}(X)) = \delta(\Phi^{\theta(X)}(X)) < \Delta_{\gamma,\kappa}(\mu_X(\theta(X)) =  \Delta_{\gamma,\kappa}(2\pi)=\Delta_{\gamma,\kappa}(0).
\end{gather*}
We have shown that $P_{\gamma,\kappa}$ is a map from $\Sigma_{\gamma,\kappa}$ into itself.
\end{proof}

\begin{corollary}\label{fixedpoint}
The Poincar\'e map $P_{\gamma,\kappa}$ defined in lemma  \ref{Poincare} admits a fixed point $X_* \in \Sigma_{\gamma,\kappa}$.
\end{corollary}

\begin{proof}
$\bar \Sigma_{\gamma,\kappa}$ is compact and convex; $P_{\gamma,\kappa} : \bar\Sigma_{\gamma,\kappa} \to \bar\Sigma_{\gamma,\kappa}$ is continuous. By Brouwer fixed point theorem, $P_{\gamma,\kappa}$ admits a fixed point in $X_* \in \bar\Sigma_{\gamma,\kappa}$. We claim  that $X_* \not\in \partial \bar\Sigma_{\gamma,\kappa}$.  Suppose by contradiction $X_* \in \partial \bar\Sigma_{\gamma,\kappa}$. Let be
\[
\delta_{i,j}^*(t) := x_i^*(t) -x_j^*(t), \quad \delta_*(t) := \max_{i,j}\delta_{i,j}^*(t) \ \ \text{and} \ \  \mu_*(t) := \frac{1}{N}\sum_{i=1}^N x_i^*(t).
\] 
Then there  exist $1 \leq i_0,j_0 \leq N$, such that  $\delta_{i_0,j_0}^*(0)=\Delta_{\gamma,\kappa}(0)$. As in the proof of item 1 of Theorem \ref{mainresult},   there exists $t_0>0$ small enough such that,  for every $0<t<t_0$, $\delta_{i_0,j_0}^*(t) < \Delta_{\gamma,\kappa}(\mu_*(t))$.  By repeating this argument for every $1 \leq i_1,j_1 \leq N$ satisfying the equality $\delta_{i_1,j_1}^*(t)=\Delta_{\gamma,\kappa}(\mu_*(t))$, we  obtain for some $t'>0$, $\delta_*(t') < \Delta_{\gamma,\kappa}(\mu_*(t'))$. The invariance of $C_{\gamma,\kappa}^N$ implies   $\delta_*(t) < \Delta_{\gamma,\kappa}(\mu_*(t))$ for every $t > t'$. We have obtained a contradiction with the fact that
\[
\delta_*(\theta(X_*)) = \delta_*(0) = \Delta_{\gamma,\kappa}(0) =   \Delta_{\gamma,\kappa}(2\pi)= \Delta_{\gamma,\kappa}(\mu_*(\theta(X_*)). \qedhere
\]
\end{proof}
\begin{proof}[Proof of Theorem \ref{mainresult} - Item 2]
Corollary \ref{fixedpoint} implies the existence of a point $X_* \in C_{\gamma,\kappa}^N$ and a time $\theta_*>0$ such that $\Phi^{\theta_*}(X_*) = X_* + 2\pi\mathbb{1}$. By the  uniqueness property of the solutions  of an ordinary differential equation, we have
\[
\Phi^{\theta_*+t}(X_*) = \Phi^t(X_*)+2\pi\mathbb{1}, \quad \forall t \geq 0.
\]
Let be   $\Psi(s) := \Phi^s(X_*) - \frac{2\pi s}{\theta_*}\mathbb{1}=(\Psi_1(s),\cdots,\Psi_N(s))$,  $\forall s\geq 0$. Then $\Psi$ is periodic of period $\theta_*$. We have indeed
\begin{align*}
\Psi(s+\theta_*) &= \Phi^{s + \theta_*}(X_*) - 2\pi \frac{s+\theta_*}{\theta_*}\mathbb{1} = \Phi^s(X_*) + 2 \pi\mathbb{1} - 2\pi \frac{s+\theta_*}{\theta_*}\mathbb{1} = \Psi(s).
\end{align*}
Moreover the return time $\theta_*$  is uniformly bounded from above with respect to $N$  as in lemma \ref{Poincare}.
\end{proof}

\section{Numerical results}
\label{section:NumericalResults}

The study of organized populations has been modeled by Winfree in \cite{WinfreeModel} in 1967. He observed that a week coupling between independent individuals tends to synchronize  them. An example of biological synchronization is given by flashes of tropical fireflies.  The synchronization appears when the individuals emit their flashes simultaneously with the same frequency. In this section we analyze a simplified Winfree model  given by equation \eqref{equation:NumericalWinfreeModel} with $P_\beta(x) = 1+\cos(x+\beta)$ and $R(x) = \sin(x)$.

 When $\beta=0$, the synchronization hypothesis  H3 is satisfied. Indeed 
\begin{gather*}
\frac{d}{ds} \Big( \ln (1-\kappa P_0 R) \Big) = \frac{-\kappa(P_{0}'R+P_{0}R')}{1-\kappa P_0R}, 
\end{gather*}
and
\begin{multline*}
\int_0^{2\pi}\!\! \frac{P_0R'}{1-\kappa P_0R} \,ds = -\int_0^{2\pi}\!\! \frac{P_0'R}{1-\kappa P_0R} \, ds \\
= \int_0^{2\pi}\!\! \frac{\sin^2(s)}{1-\kappa(1+\cos(s))\sin(s)} \,ds > \frac{\pi}{3}.
\end{multline*}
The numerical value of $ \kappa_* $ is $ \frac{4}{3\sqrt{3}} \sim 0.769$ and the two constants $C$ and  $\tilde C$ of lemmas \ref{lemma:ODEdispersion} and \ref{lemma:ODEmean} are bounded from above by $3$. The theoretical  domain of parameters $U$ for which the  Winfree model is synchronized is built in lemma \ref{lemma:definitionOpenSet} using $D(\kappa)$ and $L(\kappa)$. Its size is very small, $10^{-3}$ smaller than the size of the numerical domain one can expect as in figure \ref{031a}.

In order to analyze numerically the hypothesis H3,  we  discuss the model \eqref{equation:NumericalWinfreeModel} for different values of $\beta\in [0,\pi]$. We use  two different order parameters to measure the synchronization
\[
r_X(\beta) := \Big| \frac{1}{N} \sum_{j=1}^N e^{ix_j(T)} \Big|, \quad d_X(T) := \max_{i} |x_i(T)-\mu(T)|.
\]
Figure \ref{img031} shows the numerical domains $U_\beta$ of synchronization for three different values of $\beta$. The domain decreases along the $\gamma$-direction as $\beta$ increases to $\frac{\pi}{2}$ as we shall see. The critical parameter giving the transition to the death state is defined as in equation \eqref{equation:DeafBifurcationParameter} by $\kappa_*(\beta) := \max\{ P_\beta(x)R(x) : x \in \mathbb{R} \}^{-1}$.

\begin{figure}[h!]
\centering
\begin{subfigure}[t]{0.3\textwidth}
\centering \includegraphics[width=\textwidth]{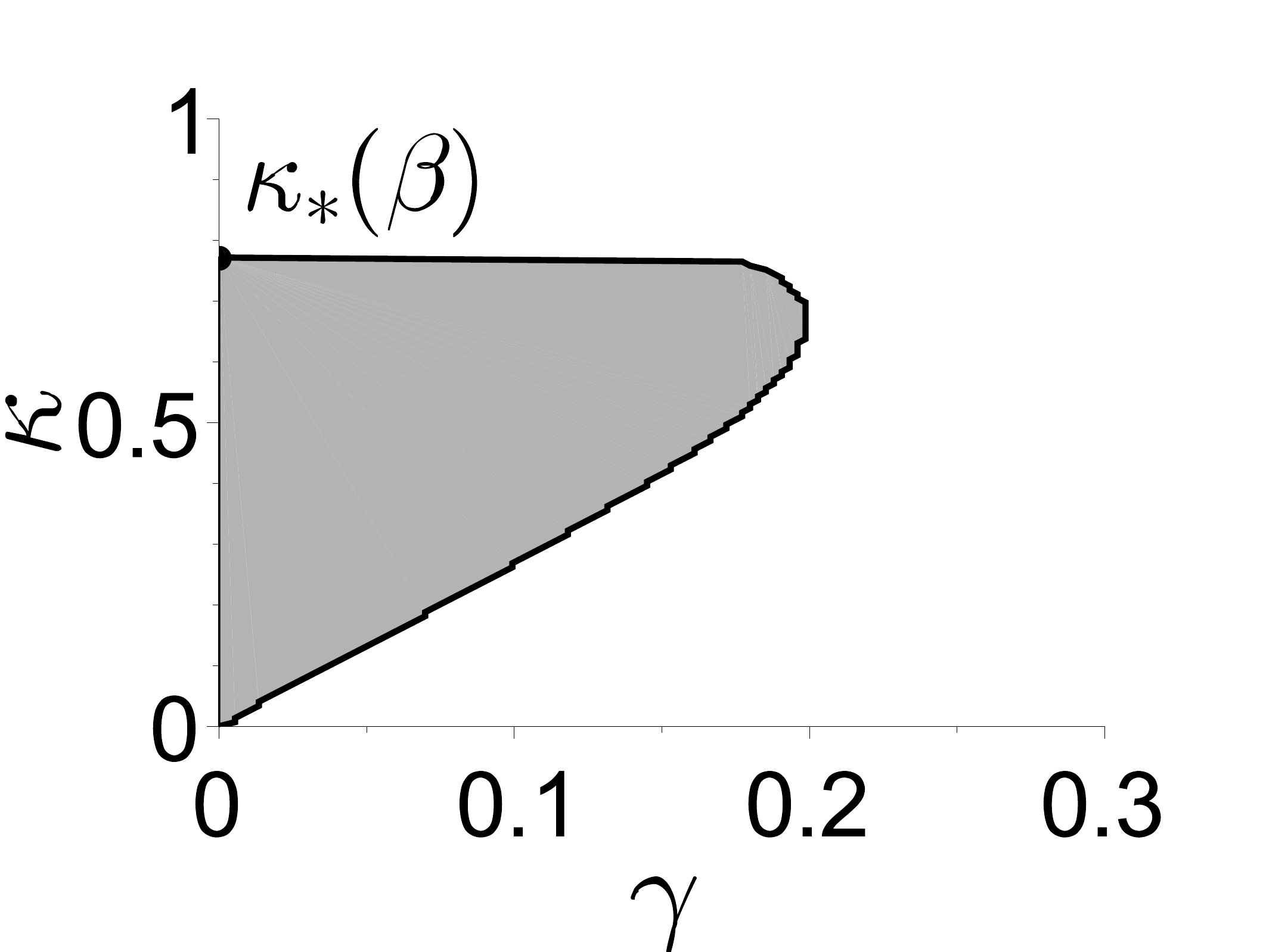}
\caption{$\beta = 0$}
\label{031a}
\end{subfigure}
~
\begin{subfigure}[t]{0.3\textwidth}
\centering \includegraphics[width=\textwidth]{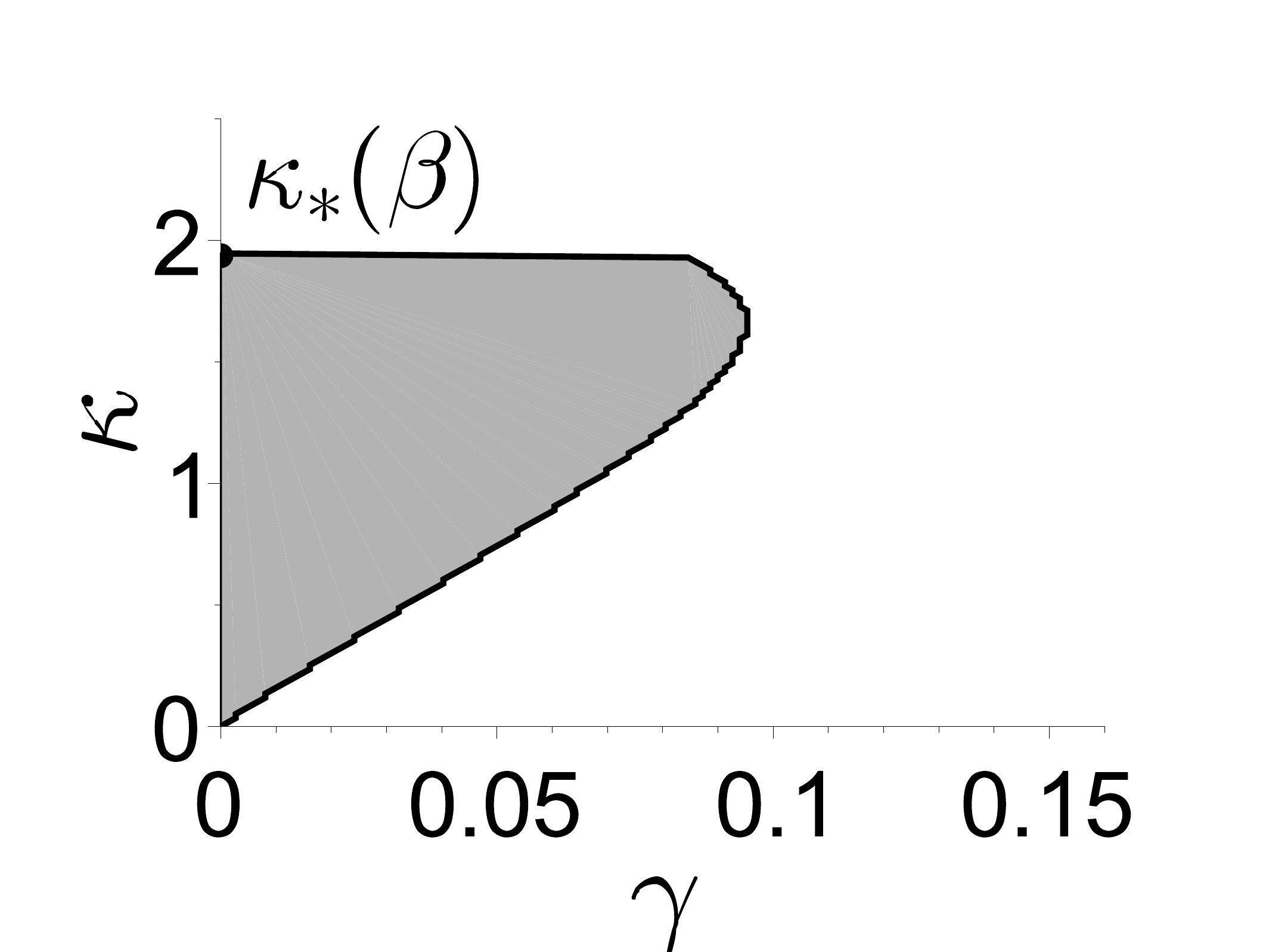}
\caption{$\beta = \frac{\pi}{2}-0.5$}
\label{031b}
\end{subfigure}
~
\begin{subfigure}[t]{0.3\textwidth}
\centering \includegraphics[width=\textwidth]{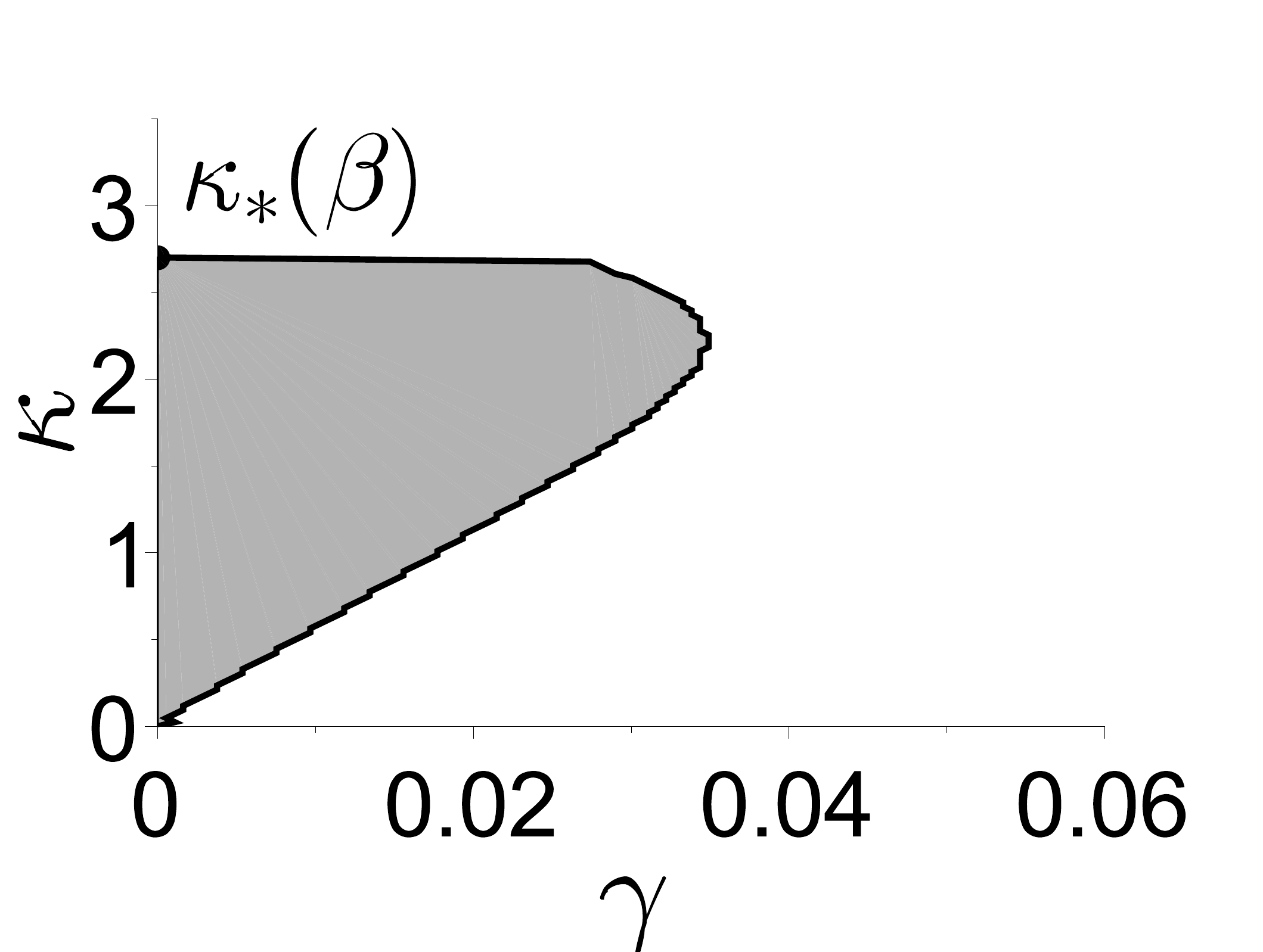}
\caption{$\beta = \frac{\pi}{2}-0.25$}
\label{031c}
\end{subfigure}
\caption{Numerical domain  of synchronization $U_\beta$ of the  model \eqref{equation:NumericalWinfreeModel}. In all cases we choose initial conditions randomly distributed in  $[-\frac{\pi}{2},\frac{\pi}{2}]$ and  $N=100$ oscillators with  ordered and equidistant natural frequencies $\omega_i$ in $[1-\gamma, 1+\gamma]$. $U_\beta$ is in grey and obtained by testing the largest $\gamma$ satisfying  $|d_X(T)|<3\pi$. In figure \subref{031a}, $T=1500$, $\beta =0$ and $\kappa_*(\beta) \approx 0.769$. In figure \subref{031b},  $T=1800$, $\beta = \frac{\pi}{2}-0.5$ and $\kappa_*(\beta)  \approx 1.936$. In figure \subref{031c}, $T=3500$, $\beta = \frac{\pi}{2}-0.25$ and $\kappa_*(\beta)  \approx 2.694$.}
\label{img031}
\end{figure}

We show in figure \ref{021a} the variation of the critical parameter $\kappa_*(\beta)$ for $\beta \in [0,\pi]$; we notice that the minimum is obtained at $\kappa_*(0) \approx 0.769$. Let be
\[
H_\kappa(\beta) := \int_0^{2\pi} \frac{ P_\beta(s)R'(s)}{1-\kappa P_\beta(s)R(s)} \, ds.
\]
We also show numerically in \ref{021a}, that the gray region corresponds to   the domain of $(\beta,\kappa)$ satisfying  $\beta \in [0, \pi]$,  $\kappa \in [0, \kappa_*(\beta)]$, and $H_\kappa(\beta) > 0$, that is to the domain  $\beta < \frac{\pi}{2}$. As noticed by the referee, the symmetry $H_\kappa(\beta) = -H_\kappa(\pi - \beta)$ implies $H_\kappa(\frac{\pi}{2})=0$ for every $\kappa \in [0,\kappa_*(\frac{\pi}{2})]$.  We choose   $\kappa=0.6$ in figure \ref{021b} and compute the largest $\gamma$ satisfying $|d_X(T)|<3\pi$, that is the boundary of $U_\beta$ at $\kappa=0.6$. We notice that the numerical domain of synchronization is negligible if and only if $\beta \in [\frac{\pi}{2}, \pi]$.
\begin{figure}[h!]
\centering
\begin{subfigure}[t]{0.48\textwidth}
 \centering \includegraphics[width=\textwidth]{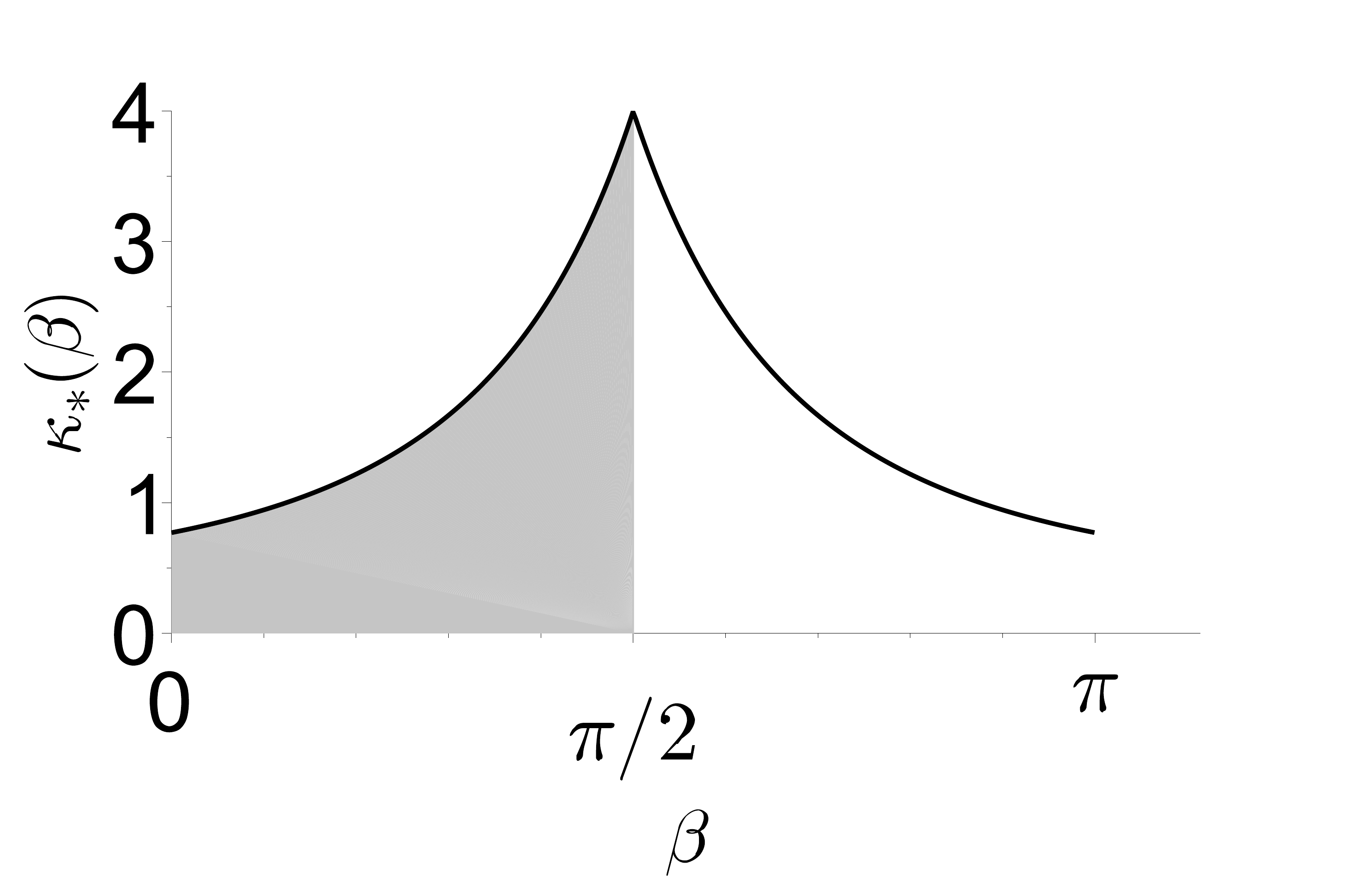}
\caption{}
\label{021a}
\end{subfigure}
~
\begin{subfigure}[t]{0.48\textwidth}
\centering \includegraphics[width=\textwidth]{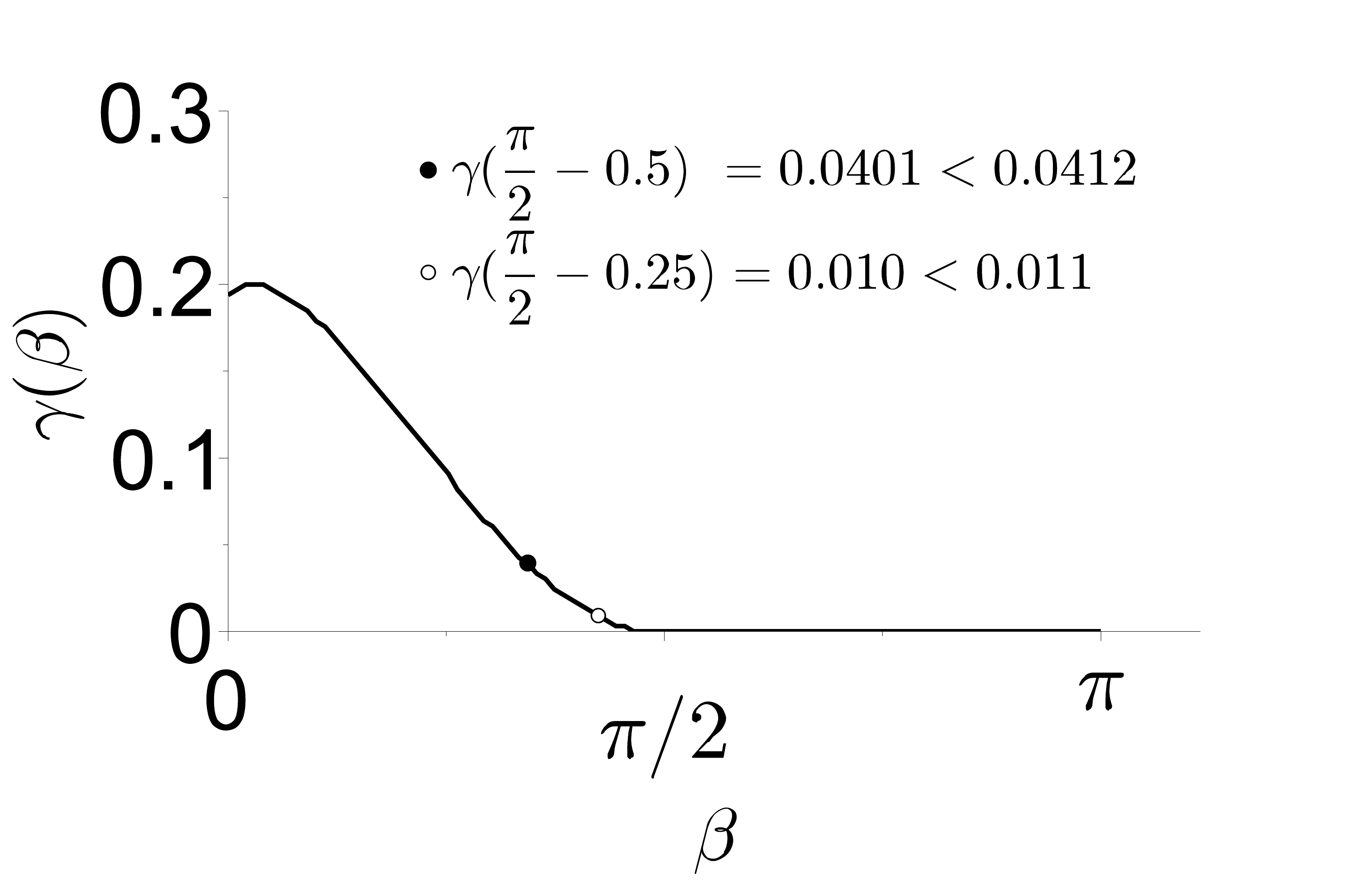}
\caption{}
\label{021b}
\end{subfigure}
\caption{We  plot in \ref{021a} the critical parameter of the death state transition $\kappa_*(\beta)$. We choose in \ref{021b}  $\kappa = 0.6$, $T= 3\times10^{4}$ and draw the {\it desynchronization curve}, the largest $\gamma$  for which $d_X(T)<3\pi$. We use the same experiment parameters as in figure \ref{img031}.}
\label{img021}
\end{figure}

We study in figure \ref{img033} the variation of the order parameter $r_X(\beta)$ as $\beta$ increases in $[0,\pi]$ at a given  time $T$ and coupling strength $\kappa$. A numerical value $r_X(\beta) \approx 1$ suggests a very tight cluster of almost all oscillators on the circle, a value $r_X(\beta) \approx 0$ suggests on the contrary symmetrically distributed oscillators. We observe a sharp decrease of $r_X(\beta)$ at $\beta=\frac{\pi}{2}$ that is when $H_\kappa(\beta)$ becomes negative.

\begin{figure}[h!]
\centering
\begin{subfigure}[t]{0.3\textwidth}
\centering \includegraphics[width=\textwidth]{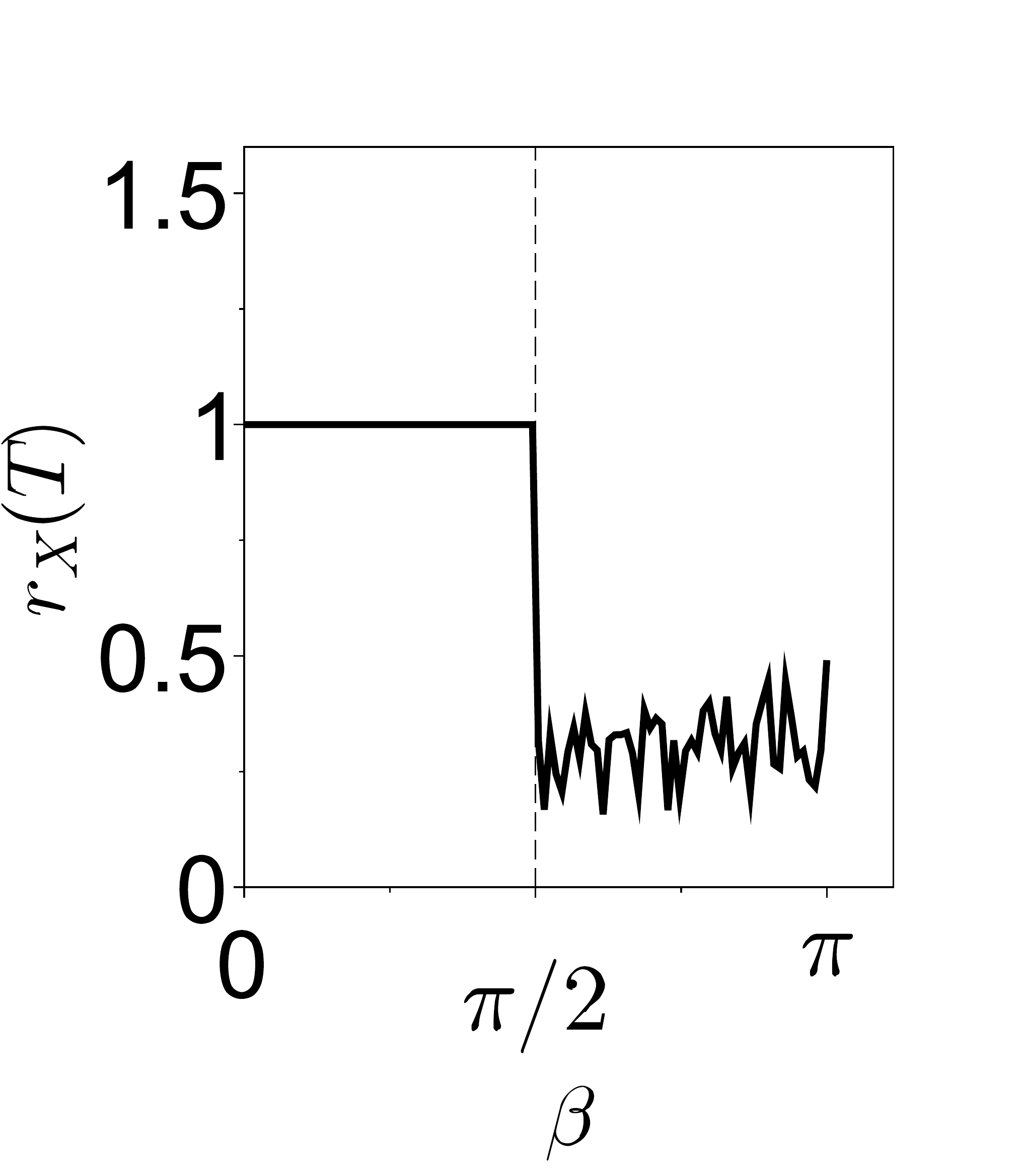}
\caption{ $\gamma_0 = 0$}
\label{033c}
\end{subfigure}
~
\begin{subfigure}[t]{0.3\textwidth}
\centering \includegraphics[width=\textwidth]{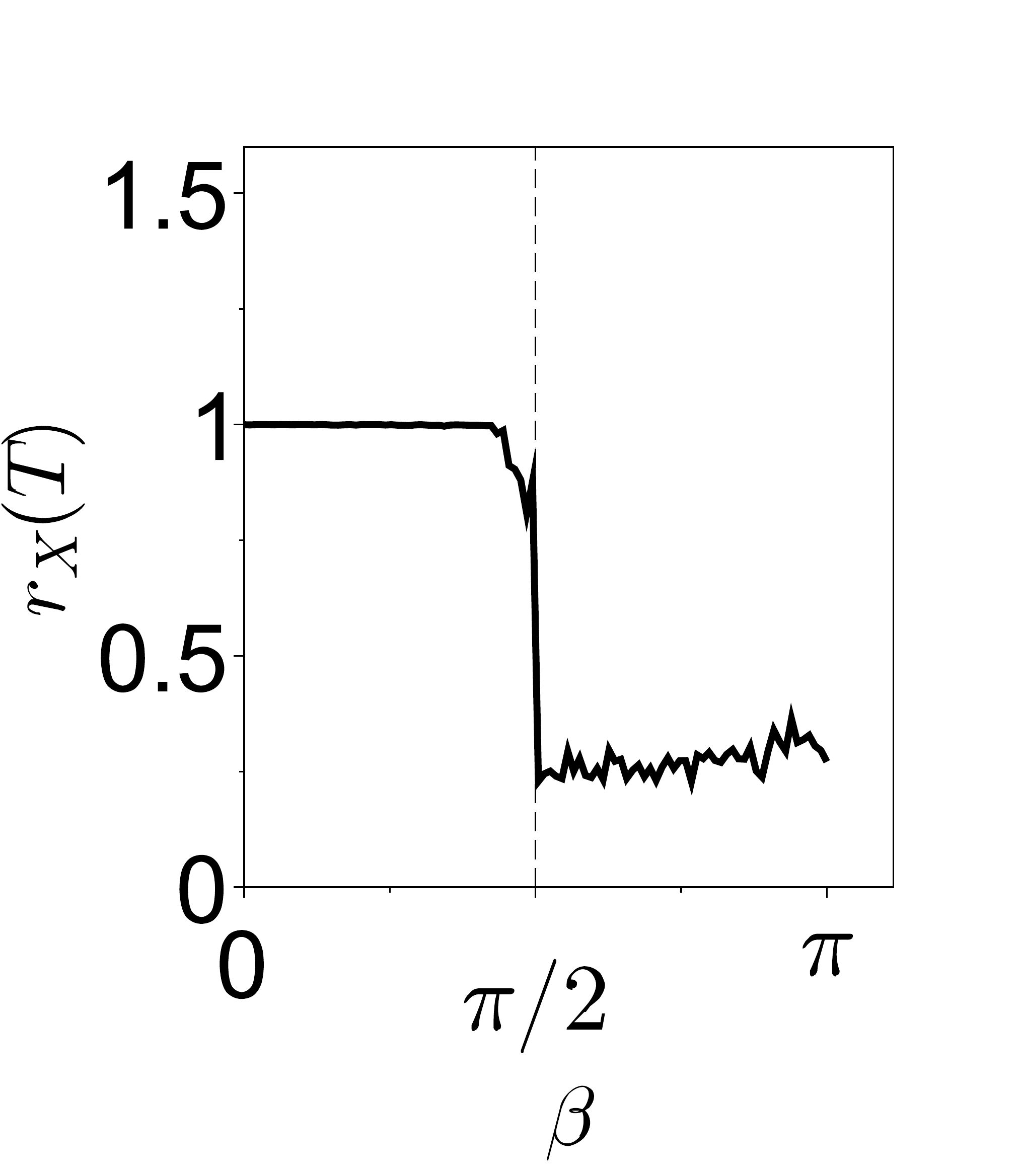}
\caption{ $\gamma_1 = 0.011$}
\label{033b}
\end{subfigure}
~
\begin{subfigure}[t]{0.3\textwidth}
\centering \includegraphics[width=\textwidth]{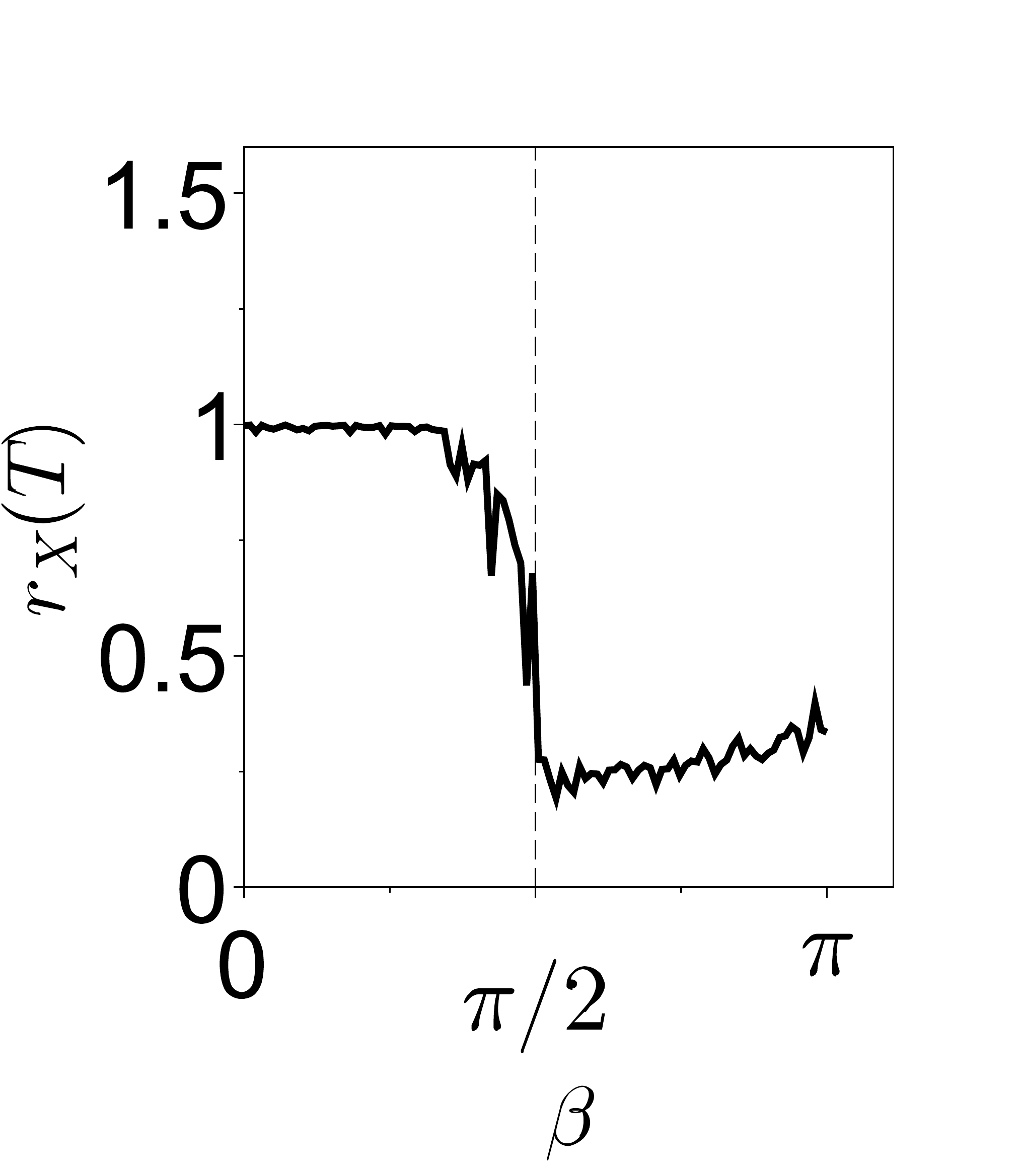}
\caption{ $\gamma_2 = 0.0412$}
\label{033a}
\end{subfigure}
\caption{The plot the order parameter $r_X(\beta)$ at $T =3000$ and  $\kappa = 0.6$. The other parameters are chosen as in figure \ref{img031}.}
\label{img033}
\end{figure}

Figure \ref{img032} shows the graph of $d_X(t)$ as a function of $t \in [0, 3 \times10^4]$ for different values of $\gamma$ and $\beta$ at a given $\kappa=0.6$. It also shows the location of the oscillators on the circle at the final time $T=3\times10^4$. Synchronization has been defined by the condition $\sup_{t>0} d_X(t) <+\infty$. For $\beta > \frac{\pi}{2}$, except at $\gamma\approx 0$, $t\mapsto d_X(t)$ does not seem  to be bounded: the oscillators are desynchronized. This observation corroborates the disappearance of the synchronization domain in figure \ref{img021} when $\beta > \frac{\pi}{2}$.   For $\gamma=0$, the oscillators have identical natural frequencies;  the comparison principle of ODE's in the periodic case implies, if $\max_{1 \leq i,j \leq N} |x_i(0)-x_j(0)| \leq 2\pi$, that $d_X(t) \leq 2\pi$ for every $t>0$: the oscillators are always synchronized. For $\beta < \frac{\pi}{2}$, we choose two values of $\gamma$ which are close, but above,  the desynchronization curve: $\gamma_1 > \gamma(\frac{\pi}{2}-0.25)$ and $\gamma_2 > \gamma(\frac{\pi}{2}-0.5)$. The plots suggest a non decreasing function $t\mapsto d_X(t)$ going to infinity  with flat parts.

\begin{figure}[h!]
    \centering
    \begin{subfigure}[t]{0.92\textwidth}
        \centering \includegraphics[width=\textwidth]{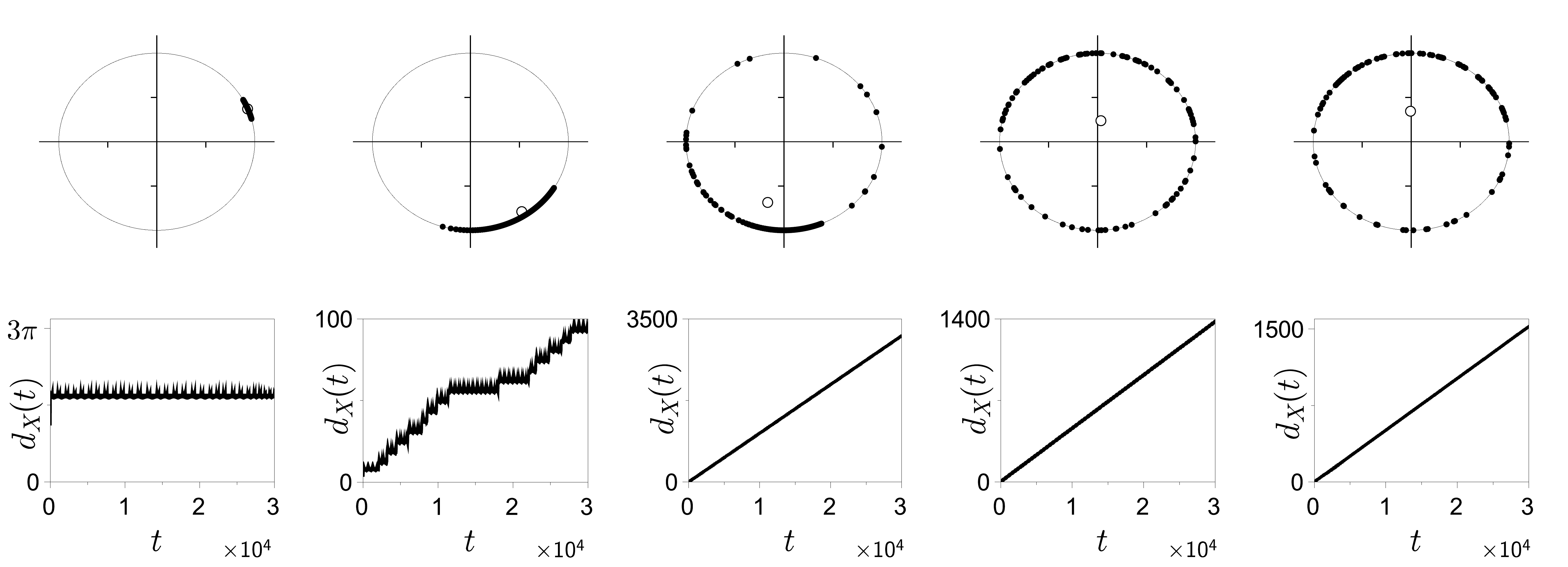}
          \caption{$\gamma_2 = 0.0412$}\nonumber
    \end{subfigure}

    \begin{subfigure}[t]{0.92\textwidth}
        \centering \includegraphics[width=\textwidth]{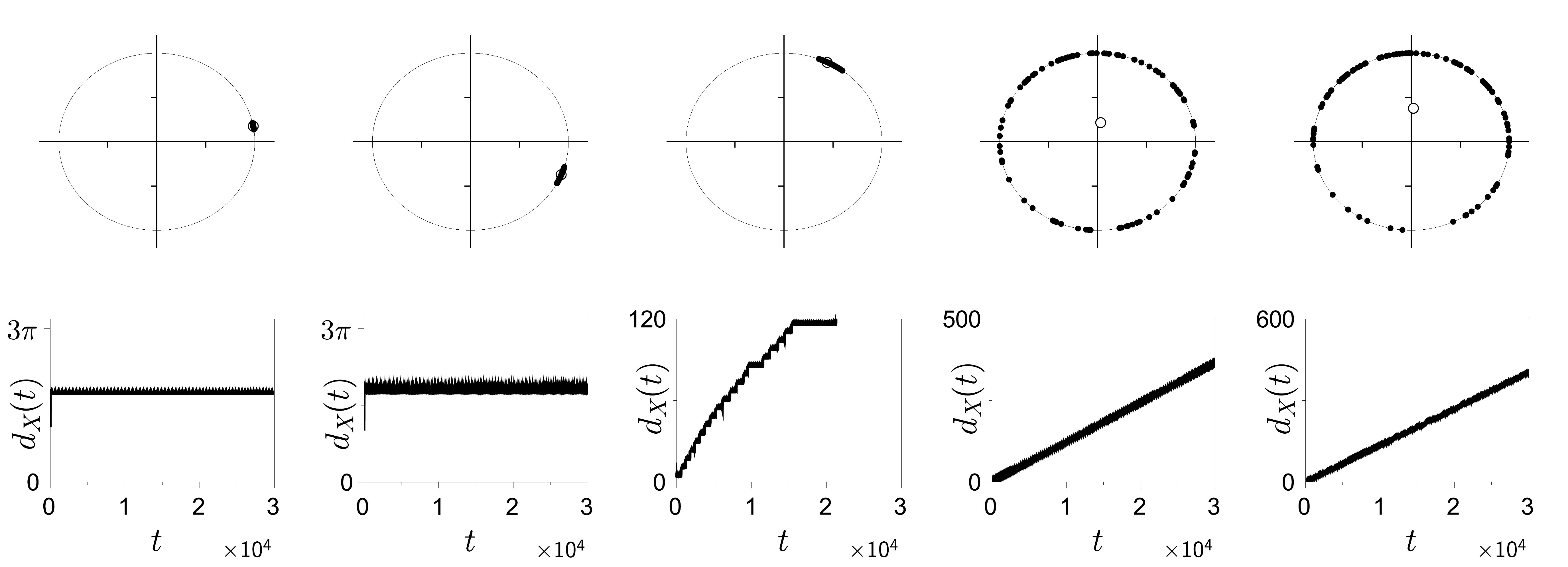}
       \nonumber\caption{$\gamma_1 = 0.011$}
    \end{subfigure}

    \begin{subfigure}[t]{0.92\textwidth}
        \centering \includegraphics[width=\textwidth]{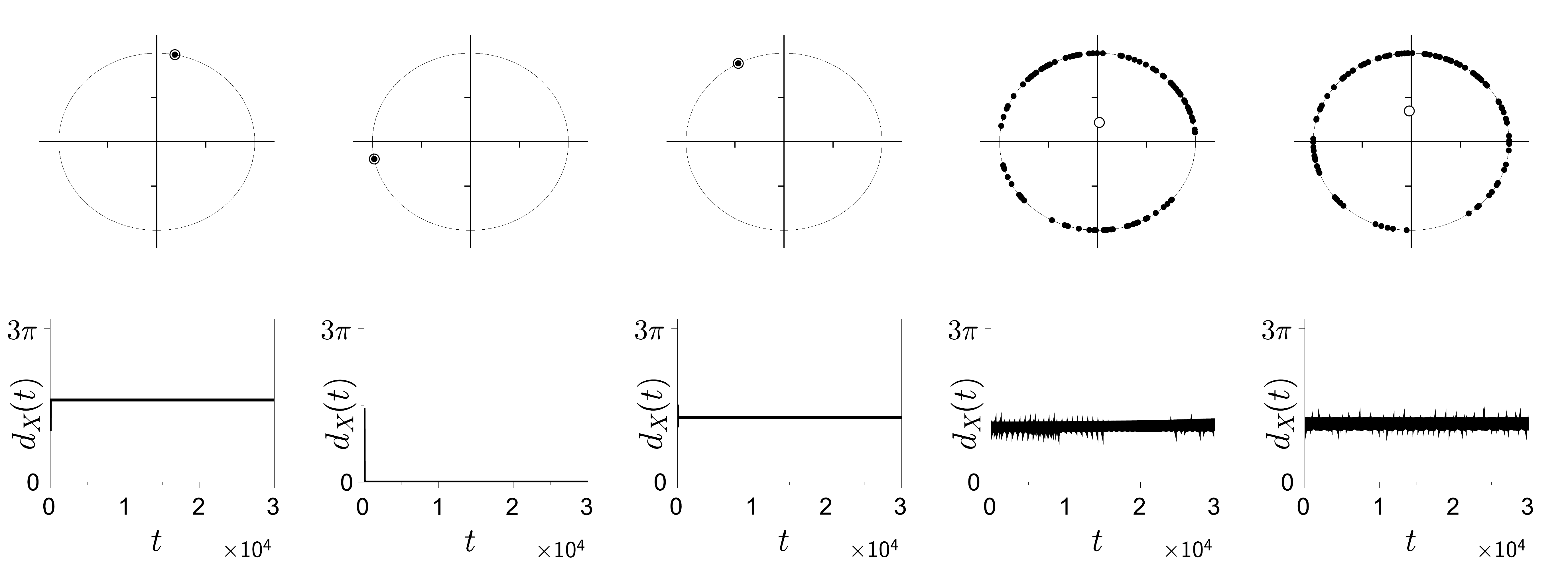}
       \caption{$\gamma_0 = 0$}
    \end{subfigure}
\caption{We choose $\kappa = 0.6$, random initial conditions  in $[-\pi,\pi]$, and the other parameters as in figure \ref{img031}.  We choose vertically from botton to top, $\gamma_0=0$,  $\gamma_1=0.011$ et $\gamma_2=0.0412$, horizontally from left to right, $\beta =0$, $\beta = \frac{\pi}{2}-0.5$, $\beta = \frac{\pi}{2}-0.25$, $\beta =\frac{\pi}{2}+0.25$ and $\beta =\frac{\pi}{2}+0.5$.   We plot the order parameter $d_X(t)$ for $t\in[0,3\times 10^{4}]$,  the oscillators on the circle (bold dots) and the order parameter $r_X(\beta)$ (circle dot) at time $T=3\times 10^{4}$.}
\label{img032}
\end{figure}

\section{Conclusion}

We defined synchronization as a state where the dispersion of the oscillators is uniformly bounded in time.  For every given coupling strength $\kappa$, corresponding to the non death state, we showed analytically the existence of synchronized oscillators when the hypothesis H3 is satisfied and when the spectrum width $\gamma$ is sufficiently smaller than some value depending on  $\kappa$. In addition, under the same hypothesis, we showed the existence of $N$ periodic synchronized oscillators with the same strong rotation number. We also showed numerically, for  a parametrized simplified Winfree model at a given coupling strength, that the hypothesis H3 seems to be necessary for the existence of  synchronization. 

\section*{Acknowledgements}

The first author would like to thank Professor Strogatz for his support to answering many questions by mail.  We also would like to thank both referees for their precise remarks that helped us to improve the text.



\begin{thebibliography}{99}
\addcontentsline{toc}{section}{Bibliographie}

\bibitem{AriaratnamStrogatz} J.T. Ariaratnam, and S.H. Strogatz, Phase Diagram for the Winfree Model of Coupled Nonlinear Oscillators, Phys. Rev. Lett. 86, 4278 (2001).

\bibitem{Basnarkov}  L. Basnarkov, and V. Urumov, Critical exponents of the transition from incoherence to partial oscillation death in the Winfree model, J. Stat. Mech.  P10014 (2009)

\bibitem{Giannuzzi} F. Giannuzzi, D. Marinazzo, G. Nardulli, M. Pellicoro, and S. Stramaglia, Phase diagram of a generalized Winfree model, Phys. Rev. E 75, 051104.(2007)

\bibitem{HaParkRyoo} S.-Y. Ha, J. Park, and S. W. Ryoo, Emergence of phase-locked states for the Winfree model in a large coupling regime
.  Discrete and Continuous Dynamical Systems, Series A, 35 , no.  8, 3417-3436 (2015)

\bibitem{Ha2} S-Y. Ha, D. Ko, J. Park, and S. W. Ryoo, Emergent dynamics of Winfree oscillators on locally coupled networks,  Journal of Differential Equations, 260, no.  5, 4203 - 4236. (2016) 

\bibitem{Kuramoto} Y. Kuramoto and D. Battogtokh, Coexistence of coherence and incoherence in nonlocally coupled phase oscillators, Nonlinear Phenom. Complex Syst. 5, 380 (2002)

\bibitem{Louca} S, Louca, and F. M. Atay, Spatially structured networks of pulse-coupled phase oscillators on metric spaces AIMS Journal of Discrete and Continuous Dynamical Systems - A, vol. 34 (2014)

\bibitem{Omelchenko} O.E. Omel’chenko, Coherence-incoherence patterns in a ring of non-locally coupled phase oscillators, Nonlinearity 26, 2469–2498 (2013)

\bibitem{Panaggio} M. J.  Panaggio, and D. M. Abrams, Chimera states: coexistence of coherence and incoherence in networks of coupled oscillators. Nonlinearity, 28, p. R67 (2015)

\bibitem{PazoMontbrio} D. Paz\'o and  E. Montbri\'o, Low-dimensional dynamics of populations of pulse-coupled oscillators, Phys. Rev. X 4  011009 (2014).

\bibitem{Popovych} O.V. Popovych, YL Maistrenko, PA Tass. Phase chaos in coupled oscillators. Physical Review E 71, 065201 (R),  (2005)

\bibitem{Quinn} D.D. Quinn, R. H. Rand and S.H. Strogatz, Singular unlocking transition in the Winfree model of coupled oscillators. Physical Review E 75, 036218 (2007)

\bibitem{WinfreeModel} A. T. Winfree, Biological rhythms and the behavior of populations of coupled oscillators {\it J. Theor. Biol.} {\bf 16} 15-42.(1967)
\end{thebibliography}
\end{document}